\documentclass{amsart}
\usepackage[utf8]{inputenc}

\usepackage{amssymb}
\def\eps{\varepsilon }

\def\RR{\mathbb R}

\def\ZZ{\mathbb Z}

\newcommand{\remove}[1]{ }

\providecommand{\abs}[1]{\lvert#1\rvert}

\newcommand{\set}[1]{\left\{#1\right\}}

\newtheorem{theorem}{Theorem}[section]
\newtheorem{proposition}[theorem]{Proposition}
\newtheorem{lemma}[theorem]{Lemma}

\theoremstyle{definition}

\theoremstyle{remark}
\newtheorem{remark}[theorem]{Remark}

\newtheorem{examples}[theorem]{Examples}

\numberwithin{equation}{section}

\begin{document}
\title{Discrete spectra and Pisot numbers}
\author{Shigeki Akiyama}
\address{Department of Mathematics, Faculty of Science, Niigata University, Ikarashi-2 8050
Niigata, 950-2181 Japan}
\email{akiyama@math.sc.niigata-u.ac.jp}
\author{Vilmos Komornik}
\address{D\'epartement de math\'ematique\\
         Universit\'e de Strasbourg\\
         7 rue Ren\'e Descartes\\
         67084 Strasbourg Cedex, France}
\email{komornik@math.unistra.fr}

\date{Version of 2011-03-23-a}
\thanks{This work was initiated during the Workshop on Numeration, Lorentz Center, Leiden, June 7--18, 2010. The authors thank the organizers for their invitation and for the excellent working conditions.}
\begin{abstract}
By the \emph{$m$-spectrum} of a real number $q>1$ we mean the set $Y^m(q)$ of values $p(q)$ where $p$ runs over the height $m$ polynomials with integer coefficients. These sets have been extensively investigated during the last fifty years because of their intimate connections with 
infinite Bernoulli convolutions, spectral properties of substitutive point sets
 and expansions in noninteger bases. 
We prove that $Y^m(q)$ has an accumulation point if and only if $q<m+1$ and $q$ is not a Pisot number. Consequently a number of related results on the distribution of points of this form are improved. 
\end{abstract}

\maketitle


\section{Introduction}\label{s1}

Given a real number $q>1$ and a positive integer $m$, we are interested
in the structure of the set
\begin{equation*}
X^m(q):= \left\lbrace \sum_{i=0}^n s_i q^i \ |\ s_i\in \set{0,1,\ldots, m},  n=0,1,\ldots  \right\rbrace.
\end{equation*}
This is an unbounded set of nonnegative real numbers and 
any bounded interval contains only finitely many elements of $X^m(q)$, i.e.,
it is closed and discrete in the Euclidean topology of $\RR$. Hence we may arrange the points
of $X^m(q)$ into an increasing sequence:
\begin{equation*}
0=x_0^{m}(q)<x_1^m(q)<x_2^m(q)<\cdots .
\end{equation*}
We also write $x_k=x_k^m(q)$ if there is no risk of confusion.
Our first target is a historical question initiated by Erd\H os et al.
to study
\begin{align*}
&\ell^m(q)=\liminf_n (x_{n+1}^m(q)-x_n^m(q))
\intertext{and}
&L^m(q)=\limsup_n (x_{n+1}^m(q)-x_n^m(q)),
\end{align*}
especially to characterize the values of $q$ for which they vanish.

Our main results will be based on the investigation of the topology of the intimately  related set
\begin{equation}\label{11}
Y^m(q):=X^m(q)-X^m(q) ,
\end{equation}
called the  \emph{$m$-spectrum} of $q$. We note that
\begin{equation*}
Y^m(q)=\left\lbrace \sum_{i=0}^n s_i q^i \ |\ s_i\in \set{0,\pm 1,\ldots, \pm m}, n=0,1,\ldots \right\rbrace
\end{equation*}
is the set of values $p(q)$ where $p$ runs over the height $m$ polynomials with integer coefficients. Here the height of $p$ is the maximum modulus of the coefficients.
These sets are unbounded and nested:
\begin{equation*}
X^1(q)\subset X^2(q)\subset \cdots
\quad\text{and}\quad
Y^1(q)\subset Y^2(q)\subset \cdots .
\end{equation*}
A reason for the name ``spectrum'' may come from the following observation.
A Pisot number is an algebraic integer $>1$ whose other conjugates have 
modulus strictly 
less than one. 
In this case both $X^m(q)$ and $Y^m(q)$ are uniformly discrete: there is a uniform lower bound of the distances of two adjacent points, and consequently $\ell^m(q)>0$. 
This lower bound easily follows from the fact that the absolute 
norm of the distance between two points is an integer and thus not less than one
(e.g. see Garsia \cite{Gar1962}). 
Since $Y^{2m}(q)=Y^m(q)-Y^m(q)$ is uniformly discrete again, for 
a Pisot number $q$ less than $m+1$, 
$Y^m(q)$ is a 
Meyer set having a scaling invariance $q Y^m(q)\subset Y^m(q)$. 
This gives a number theoretic construction of Meyer sets 
which play an important role 
in quasi-crystal analysis (see \cite{AkiBasFro2005}).
Moreover the interval tiling of $\RR$ generated by end points $Y^m(q)$ 
has a substitution rule, which is derived 
by a generalization of the method in Feng-Wen \cite{FengWen2002}. 
There are a lot of results on spectral properties of substitutive point sets and
their Pisot scaling constants. See, e.g.,  
Coquet--Kamae--Mend\`es France \cite{CoqKamMen1977}, 
Bombieri--Taylor \cite{BomTay1987}. Solomyak
\cite{Sol1997}, Lagarias \cite{Lag1999}, Lee--Solomyak \cite{LeeSol2008}, Meyer \cite{Mey1995}.
These results 
suggest that the Pisot property is essential for point sets like $Y^m(q)$
 to have good discreteness properties; Theorem \ref{t11} below gives another
evidence of such a phenomenon. 
There is also a not yet well understood Pisot phenomenon analogy between the 
topological properties of the spectra $Y^m(q)$ and singular
Bernoulli convolution distributions; 
see, e.g., Erd\H os \cite{Erd1939}, 
Peres--Solomyak \cite{PerSol2000}, Solomyak \cite{Sol1995}, Feng--Wang
\cite{FengWang2004}.

Many works have been devoted to the investigation of these sets, e.g., 
Bogm\'er--Horv\'ath--S\"ovegj\'art\'o \cite{BogHorSov1991}, 
Borwein--Hare \cite{BorHar2002, BorHar2003},  Drobot--McDonald \cite{DroMcd1980}, 
Erd\H os et al. \cite{ErdJooKom1990,ErdJooKom1998,ErdKom1998}, 
Feng--Wen \cite{FengWen2002}, Garth--Hare \cite{GarHar2006},
Komatsu \cite{Koma2002}, 
Sidorov--Solomyak \cite{SidSol2009}, Stankov \cite{Sta2010}, 
Zaimi \cite{Zai2007,Zai2010}.
There is a close relationship between the approximation properties of these sets and 
the existence of universal expansions in noninteger bases: see, e.g., Erd\H os et al. 
\cite{ErdKom1998}, Sidorov \cite{Sid2003b}, Dajani--de~Vries \cite{DajDeV2007}. 

Our first main result concerns the existence of accumulation points of $Y^m(q)$. Note that a set of real numbers has no accumulation points 
if and only if $Y^m(q)$ is a closed discrete set. We give a full answer 
to this question: the following theorem  completes former partial results
of Erd\H os--Komornik \cite{ErdKom1998} and Zaimi \cite{Zai2007}.

\begin{theorem}\label{t11}
$Y^m(q)$  has no accumulation points in $\RR$ if and only if $q$ is a Pisot number 
or $q\ge m+1$.
\end{theorem}

As mentioned above, $Y^m(q)$  has no accumulation points if $q$ is a Pisot number. 
The spectrum $Y^m(q)$ has no accumulation points either for any $q>m+1$ 
by the following simple reasoning due to de~Vries \cite{Dev2008}. 
We thank him for permission to include the following discussion in our paper.
If 
\begin{equation*}
y=s_0+s_1q+\cdots+ s_nq^n\in Y^{m}(q)
\end{equation*}
with $s_n\ne 0$, then
\begin{equation*}
\abs{y}\ge q^n-m(1+q+\cdots+q^{n-1})>q^n\left( 1-\frac{m}{q-1}\right).
\end{equation*}
Since the right side tends to infinity as $n\to\infty$, we conclude that no bounded interval contains infinitely many elements of $Y^{m}(q)$. 
Thus the difficulty of Theorem \ref{t11} lies in showing the converse.

If $\alpha$ is an accumulation point of $Y^m(q)$ then $\pm q\alpha\pm 1$ are also accumulation points of $Y^m(q)$. Indeed, if $(p_n)$ is a sequence of height $m$ polynomials satisfying $p_n(q)\to\alpha$ and $p_n(q)\ne\alpha$ for all $n$, then setting $\hat p_n(q):=\pm qp_n(q)\pm 1$ we obtain another sequence of height $m$ polynomials satisfying $\hat p_n(q)\to\pm q\alpha\pm 1$ and $\hat p_n(q)\ne\pm q\alpha\pm 1$ for all $n$. Hence the set of accumulation points of $Y^m(q)$ is either empty or unbounded.

The proof of Theorem \ref{t11} is completed by the method 
of Erd\H os-Komornik
\cite{ErdKom1998} with the addition of Proposition \ref{p32}
whose proof is elementary but involved. Since it turned out 
that minor errors in \cite{ErdKom1998} prevent us from an easy access to this problem, 
we have tried to make clear all procedures in the present article. 

Our result improves Theorem I and Proposition 2.3 in \cite{ErdKom1998} where for non-Pisot numbers the following implications were established:
\begin{equation*}
q\ge 2m+1\Longrightarrow Y^m(q)\text{ has no accumulation points } \Longrightarrow q>
\frac{m+\sqrt{m^2+4}}2.
\end{equation*}
Theorem \ref{t11} also confirms in particular a conjecture of Borwein--Hare \cite{BorHar2002}: $Y^1(q)$ has no accumulation points for $1<q<2$ if and only if $q$ is Pisot.

Based on Theorem \ref{t11}, we shall obtain a number of new results on
$\ell^m(q)$ and $L^m(q)$.
Recall that we obviously have 
$L^m(q)\ge \ell^m(q)$ for all $m$, 
\begin{equation*}
\ell^1(q)\ge \ell^2(q)\ge \cdots
\quad\text{and}\quad
L^1(q)\ge L^2(q)\ge \cdots .
\end{equation*}
Furthermore, $L^m(q)=0$ if and only if $x_{k+1}^m(q)-x_k^m(q)\to 0$.
From \cite[Remark 4.2 (b)]{Kom2011} we have
\begin{equation*}
\ell^m(q)=\inf_k (x_{k+1}-x_k).
\end{equation*}
Since 
\begin{equation*}
Y^m(q)=\set{x_k-x_n\ :\ k,n=0,1,\ldots},
\end{equation*}
it follows that
\begin{equation}\label{12}
\ell^m(q)=\inf\set{\abs{y}\ :\ y\in Y^m(q), y\ne 0}
\end{equation}
and that $\ell^m(q)=0$ if and only if $0$ is an accumulation point of $Y^m(q)$.

\begin{remark}\label{r12}
An easy application of the pigeon's hole principle (see \cite[Lemma 8]{ErdJooKom1998}) shows that if $1<q<m+1$ is not a root of a height $m$ polynomial, then $\ell^m(q)=0$. Hence $\ell^m(q)=0$ for all but countably many $1<q<m+1$.
\end{remark}

Using Theorem \ref{t11} we derive several improvements of various earlier results of 
Bugeaud \cite{Bug1996}, Erd\H os et al. \cite{ErdJooKom1990,ErdJooKom1998,ErdJooSch1996,ErdKom1998}, Jo\'o--Schnitzer \cite{JooSch1996}, Peres--Solomyak \cite{PerSol2000}. 

We recall from \cite[Theorem II, Lemma 2.1 (b) and Proposition 3.3]{ErdKom1998} the following properties:

\begin{itemize}
\item[(i)] If $q\ge m+1$, then $\ell^m(q)>0$.
\item[(ii)] If $q\ge (m+\sqrt{m^2+4})/2$, then $L^m(q)\ge 1$ and thus $L^m(q)>0$.
\item[(iii)] If $q$ is a Pisot number, then $\ell^m(q)>0$ and $L^m(q)>0$ for all $m$.
\end{itemize}

The following theorem goes in the opposite direction; it improves \cite[Theorems II and III]{ErdKom1998} where stronger conditions were imposed on $m$:

\begin{theorem}\label{t13}
If $q>1$ is not a Pisot number, then $\ell^{2m}(q)=L^{3m}(q)=0$ for all $m>q-1$.
\end{theorem}

It is interesting to compare this result with Theorem 6 in Meyer \cite{Mey1995} (see also Lagarias \cite[Theorem 4.1 (iii)]{Lag1999}): if $Y$ is a relatively dense set such that $Y-Y$ is uniformly discrete and $qY\subset Y$ for some real number $q>1$, then $q$ is an algebraic integer all of whose algebraic conjugates $p$ satisfy $\abs{p}\le 1 $. Applying Meyer's theorem for $Y=Y^m(q)$ we obtain that if 
$q<m+1$ and $\ell^{4m}(q)>0$, then $q$ is either a Pisot or a Salem number. (Relative density means that there exists a number $R>0$ such that each real interval of length $R$ meets $Y$.) Since our scope is restricted to  point sets of a special shape, by 
Theorem \ref{t13} we could infer from the weaker condition 
$\ell^{2m}(q)>0$ the stronger conclusion that $q$ is a Pisot number.

Theorem \ref{t13} may be sharpened for $1<q<2$:

\begin{theorem}
\label{t14}
Let $1<q<2$ be a non-Pisot number.
\mbox{}

\begin{itemize}
\item[(i)] If $1<q\le \sqrt[3]{2}\approx 1.26$, then $L^1(q)=0$.
\item[(ii)] If $1<q\le \sqrt{2}\approx 1.41$, then $\ell^{1}(q)=L^2(q)=0$.
\item[(iii)] If $1<q<2$, then $\ell^{2}(q)=L^{3}(q)=0$.
\end{itemize}
\end{theorem}

Part (i) improves \cite[Theorem IV]{ErdKom1998} where the relation $L^1(q)=0$ was established for all $1<q\le 2^{1/4}\approx 1.19$, except possibly $\sqrt{P_2}\approx 1.17$, where $P_2\approx 1.38$ is the second Pisot number. 
This implied the existence of universal expansions 
in these bases (\cite[Theorem V]{ErdKom1998}). We recall that an expansion 
\begin{equation*}
x=\sum_{i=1}^{\infty}\frac{c_i}{q^i}
\end{equation*}
is \emph{universal} if the sequence $(c_i)$ contains all possible finite blocks of digits.
Sidorov--Solomyak in \cite[Examples 3.1, 3.5 and 3.7]{SidSol2009}
showed $L^1(q)=0$ 
for the positive roots of the equations $q^4=q+1$, $q^{12}=q^9+q^6+q^3-1$ and
\begin{equation*}
q^{18}=-q^{16}+q^{14}+q^{11}+q^{10}+\cdots +q+1
\end{equation*}
(approximately equal to $1.22$, $1.19$ and $1.22$). Part (i) gives a unified proof
because they are not Pisot numbers and are smaller than $2^{1/3}$. 

Part (ii) solves in particular Problem 7 of Jo\'o--Schnitzer \cite{JooSch1996}: we have $\ell^{1}(q)=0$ for all $1<q<P_1$ 
where $P_1\approx 1.32$ is the first Pisot number. We recall 
from \cite[Proposition 9]{ErdJooKom1998} that for $q=\sqrt{2}$ we have even $L^1(q)=0$.
Part (iii) is just the special case of Theorem \ref{t13} for $m=1$. 
This improves two earlier theorems obtained respectively by Erd\H os--Jo\'o--Schnitzer \cite{ErdJooSch1996} and in \cite{ErdKom1998}:
\begin{align*} 
&\text{if $1<q<(1+\sqrt{5})/2$ is not a Pisot number, then $\ell^{2}(q)=0$;}\\
&\text{if $1<q<2$ is not a Pisot number, then $\ell^{3}(q)=0$.}
\end{align*}

Our final result, based on Theorem \ref{t14} (ii), is an improvement of Proposition 1.2 of Peres--Solomyak \cite{PerSol2000}. Given $1<q\le 2$, let us denote by $A(q)$ the set of real numbers of the form
\begin{equation*}
\sum_{i=0}^{n}a_iq^i,\quad n=0,1,\ldots,\quad a_0,\ldots,a_n\in\set{-1,1}.
\end{equation*}

\begin{theorem}\label{t15}
If $1<q\le \sqrt{2}$ is not a Pisot number, then $A(q)$ is dense in $\RR$.
\end{theorem}

Similarly to $Y^m(q)$, if $1<q\le 2$ is a Pisot number, 
then $A(q)$ is uniformly discrete and hence has no accumulation points.
Peres and Solomyak proved the density of $A(q)$ for $1<q\le \sqrt{2}$ under the stronger assumption that $q^2$ is not a root 
of a polynomial with coefficients in $\set{-1,0,1}$. They also proved 
that $A(q)$ is dense in $\RR$ for almost all $\sqrt{2}<q<2$.
Borwein--Hare \cite{BorHar2002} found non-Pisot numbers $\sqrt{2}<q<2$ for which $A(q)$ has no accumulation points.
Stankov \cite{Sta2010} proved that the last property may occur only for Perron numbers (algebraic numbers with all conjugates of modulus less than $q$).
More generally, it follows from a recent theorem of Sidorov--Solomyak \cite{SidSol2009} (see Theorem \ref{t54} below) that if $A(q)$ is not dense in $\RR$, then $q$ is an algebraic integer and its conjugates $p$ satisfy the relations
\begin{equation*}
q^{-1}<\abs{p}<q\quad\text{or}\quad p=\pm q^{-1}.
\end{equation*}

The initial motivation for our work was to treat the exceptional case of the second Pisot number in \cite[Theorem IV]{ErdKom1998}, by proving the following

\begin{proposition}\label{p16}
We have $L^1(\sqrt{P_2})=0$.
\end{proposition}

Since this was an open question for a long time, we give in Section \ref{s2} a short proof without using Theorem \ref{t11}. For this we combine a generalization of a result of \cite{ErdJooKom1998}, a recent theorem of Sidorov--Solomyak \cite{SidSol2009} and an elementary property of Pisot numbers.

Theorems \ref{t11}, \ref{t13}, \ref{t14} and \ref{t15} will be proved in Sections \ref{s3}, \ref{s4}, \ref{s5} and \ref{s6}, respectively. At the end of the paper we list some open problems.

\section{Proof of Proposition \ref{p16}}\label{s2}

First we recall \cite[Lemmas 2.1]{ErdKom1998}:

\begin{lemma}\label{l21}
Fix a real number $q>1$, a positive integer $m$ and consider the sequence $(x_k):=(x^m_k(q))$.

\begin{itemize}
\item[(i)] If  $q \le m+1$, then $x_{k+1}-x_k\le 1$ for all $k$.
\item[(ii)] If $q \ge m+1$, then $x_{k+1}-x_k\ge 1$ for all $k$, and $x_{k+1}-x_k=1$ for infinitely many $k$, so that $\ell^m(q)=1$.
\end{itemize}
\end{lemma}

The following lemma was proved explicitly in \cite[Lemma 7]{ErdJooKom1998} for $m=1$ and implicitly during the proof of \cite[Lemmas 3.2]{ErdKom1998} for all $m$:

\begin{lemma}\label{l22}
Fix again a real number $q>1$, a positive integer $m$ and consider the sequence $(x_k):=(x^m_k(q))$.

If $\ell^m(q)=0$, then for any fixed positive numbers $\delta<\Delta$ there exist indices $n<N$ such that $x_{k+1}-x_k<\delta$ for all $k=n,n+1,\ldots, N-1$, and $x_N-x_n>\Delta$.
\end{lemma}

\begin{remark}\label{r23}
More generally, the conclusion remains valid by the same proof if $(x_k)$ is a subsequence of $(x^m_k(q))$ satisfying the following two conditions:

\begin{itemize}
\item $\liminf (x_{k+1}-x_k)=0$;
\item $(q^Nx_k)$ is a subsequence of $(x_k)$ for some positive integer power $q^N$ of $q$.
\end{itemize}
\end{remark}

The next simple result was established implicitly in the proofs of \cite[Theorem 5]{ErdJooKom1998} (for $m=1$) and of \cite[Lemma 3.2]{ErdKom1998} (for all $m$). For the reader's convenience we recall the short proof.

\begin{lemma}\label{l24}
Let $(u_k),(v_k),(z_k)$ be three strictly increasing sequences of real numbers such that each sum $u_k+v_{\ell}$ appears in $(z_n)$. Assume that

\begin{itemize}
\item for any fixed positive numbers $\delta<\Delta$ there exist indices $n<N$ such that $u_{k+1}-u_k<\delta$ for all $k=n,n+1,\ldots, N-1$ and $u_N-u_n>\Delta$.
\item $v_k\to\infty$ but the difference sequence $(v_{k+1}-v_k)$ is bounded from above.
\end{itemize}

Then $z_{k+1}-z_k\to 0$.
\end{lemma}

\begin{proof}
Let $\Delta$ be an upper bound of the difference sequence $(v_{k+1}-v_k)$. For any fixed $0<\delta<\Delta$ choose $n<N$ according to the first condition. Then every interval $(a,a+\delta)$ with $a>v_0+u_n$ contains at least one element of the sequence $(z_k)$ and therefore  $z_{k+1}-z_k\to 0$.
\end{proof}

The following lemma improves and generalizes \cite[Theorem 5]{ErdJooKom1998}: 

\begin{lemma}\label{l25}
If $\ell^m(q^r)=0$ for some integer $r\ge 2$, then $L^m(q)=0$.
\end{lemma}

\begin{proof}
It follows from Lemma \ref{l21} (ii) that $q^r<m+1$. Therefore we may apply Lemma \ref{l24} with $(u_k)=(x^m_k(q^r))$, $(v_k)=(qu_k)$ and $(z_k)=(x^m_k(q))$. The hypotheses for $(u_k)$ and $(v_k)$ follow by using the properties $\ell^m(q^r)=0$, $q^r<m+1$ and by applying Lemma \ref{l22} and Lemma \ref{l21} (i) with $q^r$ instead of $q$.
\end{proof}

We also need Corollary 2.3 in Sidorov--Solomyak \cite{SidSol2009}:

\begin{lemma}\label{l26}
If $1<q<2$ is not Pisot but $q^2$ is a Pisot number, then $\ell^1(q)=0$.
\end{lemma}

Finally, we need the following elementary algebraic result:

\begin{lemma}\label{l27}
If $q^r$ and $q^s$ are Pisot numbers for $q>1$ and some positive integers $r,s$, then $q^t$ is also a Pisot number where $t$ denotes the greatest  common divisor of $r$ and $s$. 
\end{lemma}

\begin{proof}
We may assume without loss of generality that $r,s>1$ and $t=1$. For a positive integer $k$, we denote by $f_k(x)$ the minimal  polynomial of $q^k$ and put $\zeta_k=\exp(2\pi \sqrt{-1}/k)$. Then $q \zeta_r^{i}$ for $i=0,1,\dots, r-1$  are all the roots of $f_r(x^r)$ of modulus $\ge 1$,  and an analogous property holds for $s$. 

If $p$ is a conjugate of $q$ satisfying $\abs{p}\ge 1$, then it follows that 
\begin{equation*}
p\in \set{ q \zeta_r^{i} \ |\ i=0,1,\dots, r-1 }\cap \set{ q\zeta_s^{i}\ |\ i=0,1,\dots, s-1 }.
\end{equation*}
Since $r$ and $s$ are relatively prime, the last intersection is equal to $\set{q}$ and hence $p=q$. 
\end{proof}

\begin{proof}[Proof of Proposition \ref{p16}]
Setting $q=\sqrt{P_2}\approx 1.17$, by Lemma \ref{l25} it is sufficient to prove that $\ell^1(q^3)=0$. Furthermore, since $1<q^3<2$, by Lemma \ref{l26} it suffices to show that $q^6$ is Pisot but $q^3$ is not Pisot. The first property is obvious because $q^6$ is a power of the Pisot number  $q^2$. 

Since $q$ is smaller than the first Pisot number $P_1\approx 1.32$, it is not a Pisot number. Since $q^2$ is a Pisot number by assumption, applying Lemma \ref{l27} we conclude that 
$q^3$ is not a Pisot number. 
\end{proof}

\section{Proof of Theorem \ref{t11}}\label{s3}

In view of our remarks in the introduction it remains to show that if $1<q<m+1$ and $q$ is not a Pisot number, then $Y^m(q)$  has an accumulation point in $\RR$.  By diminishing $m$ if needed we may and will assume henceforth that $m<q< m+1$.

As in \cite{ErdKom1998}, we need the following results (Lemmas 1.3 and 1.4 of that paper):

\begin{lemma}\label{l31}
Let $m<q< m+1$ and assume that $Y^m(q)$  has no accumulation points.

\begin{itemize}
\item[(i)] Then $q$ is an algebraic integer.
\item[(ii)] Let $(s_i)$ be a sequence of integers satisfying $\abs{s_i}\le m$ for all $i$, such that
\begin{equation*}
\sum_{i=0}^{\infty}s_iq^{-i}=0,
\end{equation*}
and let $p$ be an algebraic conjugate of $q$. Then 
\begin{equation*}
\sum_{i=0}^{\infty}s_ip^{-i}=0
\end{equation*}
if $\abs{p}>1$, and the partial sums
\begin{equation*}
\sum_{i=0}^ns_ip^{-i},\quad n=0,1,\ldots
\end{equation*}
belong to finitely many circles centered at the origin if $\abs{p}=1$.
\end{itemize}
\end{lemma}

In view of this lemma the proof of Theorem \ref{t11} will be completed by establishing the following improvement of \cite[Lemma 1.5]{ErdKom1998}, where we replace the former assumption $q\le \frac{m+\sqrt{m^2+4}}2$ by the weaker condition $q\le m+1$:

\begin{proposition}\label{p32}
Let $m$ be a positive integer, $m<q<m+1$, and $p\ne q$ a complex number. 

\begin{itemize}
\item[(a)] If $\abs{p}>1$, then there exists a sequence $(s_i)$ of integers satisfying $\abs{s_i}\le m$ for all $i$, and such that $s_0=-1$, 
\begin{equation*}
\sum_{i=0}^{\infty}s_iq^{-i}=0
\quad\text{and}\quad
\sum_{i=0}^{\infty}s_ip^{-i}\ne 0.
\end{equation*}

\item[(b)] If $\abs{p}=1$, then there exists a sequence $(s_i)$ of integers satisfying $\abs{s_i}\le m$ for all $i$, and such that $s_0=-1$, 
\begin{equation*}
\sum_{i=0}^{\infty}s_iq^{-i}=0,
\end{equation*}
and the sequence 
\begin{equation*}
\left|\sum_{i=0}^{n}s_ip^{-i}\right|,\quad n=0,1,\ldots
\end{equation*}
takes infinitely many different values.
\end{itemize}
\end{proposition}

\begin{remark}\label{r33}
The last property is obviously satisfied if the partial sums of the series $\sum s_ip^{-i}$ are unbounded.
\end{remark}

For the proof of Proposition \ref{p32} we need two technical lemmas. 
The first one is \cite[Lemma 1.8]{ErdKom1998} which is an 
application of the ``lazy'' expansion algorithm. 

\begin{lemma}\label{l34}
Fix a real number $q>1$ and an integer $m>q-1$. Let $P$ be a set of positive integers such that 
\begin{equation}\label{31}
1\le m\sum_{i\in P}q^{-i}.
\end{equation}
Then there exists  a sequence $(s_i)$ of integers satisfying the following conditions:
\begin{equation}\label{32}
\begin{cases}
s_0=-1;\\
s_i\in\set{0,1,\ldots, m}&\text{if $i\in P$;}\\
s_i\in\set{0,-1,\ldots, -m}&\text{if $i\notin P$;}\\
\sum_{i=0}^{\infty}s_iq^{-i}=0.
\end{cases}
\end{equation} 
\end{lemma}

\begin{remark}\label{r35}
In \cite{ErdKom1998} it was also stated that the sequence 
$(s_i)$ is \emph{infinite}, i.e., it has infinitely many nonzero elements.
We will not need this stronger conclusion hereafter and the proof in 
\cite{ErdKom1998} is correct up to the above weaker statements.
It was overlooked in the proof that some partial sums $S_k$ of the constructed 
series may vanish. We note, however, that all main results of \cite{ErdKom1998} 
remain valid by a small correction of the proofs.
\end{remark}

The second lemma expresses an interesting property of complex geometric series:

\begin{lemma}\label{l36}
Let $p\ne 1$ be a nonpositive complex number. 

\begin{itemize}
\item[(a)] If $\abs{p}\ge 1$, then there exists a complex number $w$ satisfying 
\begin{equation*}
\Re w>0,
\quad\text{and}\quad
\sum_{i=1}^k\Re (wp^{-i})\le 0
\quad\text{for all}\quad
k=0,1,\ldots .
\end{equation*}

\item[(b)]  If $\abs{p}= 1$, then for each positive integer $m$ there exists a complex number $w$ satisfying 
\begin{equation*}
\Re w>0,
\quad
m\sum_{i=1}^k\Re (wp^{-i})<\Re w
\quad\text{for all}\quad
k=0,1,\ldots ,
\end{equation*}
and 
\begin{equation*}
\Re (wp^{-i})\ne 0
\quad\text{for all but at most one index}\quad
i=0,1,\ldots .
\end{equation*}
\end{itemize}

We may assume in both cases that $w$ has modulus one.
\end{lemma}

\begin{proof}\mbox{}

(a) If $\abs{p}\ge 1$ and $\Re p<1$, then setting $w=1-p$ we have $\Re w>0$, and
\begin{equation*}
\sum_{i=1}^k\Re (wp^{-i})=\Re \left( w\frac{1-p^{-k}}{p-1}\right) =\Re (p^{-k})-1\le \abs{p^{-k}}-1\le 0
\end{equation*}
for  $k=0,1,\ldots .$
\medskip

If $\abs{p}> 1$ and $p$ is nonreal, then set $z=\pm (1-p^{-1})i$ by choosing the sign such that $\Re z>0$. Since
\begin{equation*}
\sum_{i=0}^{\infty}\Re(zp^{-i})=\Re\left(z\sum_{i=0}^{\infty}p^{-i}\right)=\Re (\pm i)=0,
\end{equation*}
and since the first partial sum is positive, there is a first maximal partial sum, say 
\begin{equation*}
\sum_{i=0}^{n}\Re(zp^{-i}).
\end{equation*}
Then we have
\begin{equation*}
\sum_{i=n+1}^{k}\Re(zp^{-i})\le 0\quad\text{for all}\quad k\ge n .
\end{equation*} 
Observe that $\Re(zp^{-n})> 0$. This follows from the choice of $z$ if $n=0$, and from the minimality of $n$ if $n>0$. Hence the  number $w:=zp^{-n}$ satisfies the requirements of the lemma.
\medskip

(b) If $(\arg p)/\pi$ is irrational, then $w=1-p$ has the required property by (a) and because  $\arg(zp^{-i})=\pi/2 \mod \pi$  may occur for at most one index $i$. Otherwise $p^n=1$ for some integer $n\ge 2$ and it suffices to observe that the finitely many strict inequalities
\begin{equation*}
m\sum_{i=1}^k\Re (wp^{-i})<\Re w,\quad k=1,\ldots, n
\end{equation*}
remain valid by slightly changing $w$.
\medskip

We conclude the proof by observing that by changing $w$ to $w/\abs{w}$ all inequalities remain valid.
\end{proof}

\begin{proof}[Proof of Proposition \ref{p32}]
We distinguish several cases.\mbox{}
\medskip

\emph{Step 1: $p\in (1,\infty)$.} We choose an arbitrary expansion $1=\sum_{i=1}^{\infty}s_iq ^{-i}$ with $(s_i)\subset\set{0,1,\ldots, m}$; this is possible because $m<q\le m+1$.  Then  setting $s_0=-1$ we get a suitable sequence because the function
\begin{equation*}
p\mapsto \sum_{i=1}^{\infty}\frac{s_i}{p^i}
\end{equation*}
is strictly monotone on $(1,\infty)$.
\medskip

\emph{Step 2: $p=1$.} We consider the sequence $(s_i)$ of Step 1.  If it is infinite, then $s_i\ge 0$ for all $i>0$ and $s_i\ge 1$ for infinitely many indices. Hence 
\begin{equation*}
\sum_{i=0}^{\infty}s_ip^{-i}=\sum_{i=0}^{\infty}s_i=\infty .
\end{equation*}

If the sequence $(s_i)$ has a last nonzero term $s_n$, then the assumption $m<q$ implies that 
\begin{equation*}
\frac{s_1+\cdots+s_n}{m}> \sum_{i=1}^{n}\frac{s_i}{q^i}=1
\end{equation*}
and therefore $s_0+\cdots+s_n>m-1\ge 0$. Hence
replacing $(s_i)$ with the periodic sequence $(s_0\ldots s_n)^{\infty}$ we have
\begin{equation*}
\sum_{i=0}^{\infty}s_iq^{-i}=0
\quad\text{and}\quad
\sum_{i=0}^{\infty}s_ip^{-i}=\sum_{i=0}^{\infty}s_i=\infty
\end{equation*}
again.
\medskip

\emph{Step 3: $p$ is nonpositive and $\abs{p}> 1$.} We choose a complex number $w$ 
with $|w|=1$ according to Lemma \ref{l36} (a). 
Next we set 
\begin{equation*}
P':=\set{i\ :\ \Re(wp^{-i})\le 0}
\end{equation*}
and we denote by $k$ the smallest nonnegative integer satisfying the inequality
\begin{equation}\label{33}
\left(\sum_{i=1}^kmq^{-i}\right)+\left(\sum_{i>k, i\in P'}mq^{-i}\right)\ge 1.
\end{equation}
Observe that $\Re(wp^{-k})>0$ (we distinguish the cases $k=0$ and $k>0$).

Applying Lemma \ref{l34} with $P:=P'\cup\set{1,\ldots, k}$ (we have $P=P'$ if $k=0$) we obtain a sequence  $(s_i)$ of integers satisfying \eqref{32}.

We have $s_j=m$ if $0< j<k$, and $1\le s_k\le m$ if $k>0$. For otherwise we would have either
\begin{equation*}
\sum_{i=1}^{k-1}mq^{-i}\ge \left(\sum_{i=1}^{k-1}s_iq^{-i} \right)+q^{-j}> \sum_{i=1}^{k}s_iq^{-i}
\end{equation*}
because $q^{k-j}\ge q>m\ge s_k$, or
\begin{equation*}
\sum_{i=1}^{k-1}mq^{-i}\ge \sum_{i=1}^{k-1}s_iq^{-i}=\sum_{i=1}^{k}s_iq^{-i}.
\end{equation*}
Since $k\notin P'$, in both cases we would conclude that
\begin{align*}
\left(\sum_{i=1}^{k-1}mq^{-i}\right)+\left(\sum_{i>k-1, i\in P'}mq^{-i}\right)
&\ge \left(\sum_{i=1}^{k}s_iq^{-i}\right)+\left(\sum_{i>k-1, i\in P'}mq^{-i}\right)\\
&= \left(\sum_{i=1}^{k}s_iq^{-i}\right)+\left(\sum_{i>k, i\in P'}mq^{-i}\right)\\
&\ge \sum_{i=1}^{\infty}s_iq^{-i}\\
&=-s_0\\
&=1,
\end{align*}
contradicting the minimality of $k$.

Since $\Re(wp^{-k})>0$, using Lemma \ref{l36} (a) we have
\begin{align*}
\Re\left(w\sum_{i=0}^{\infty}s_ip^{-i}\right)
&=-\Re (w)+m\left(\sum_{i=1}^{k-1}\Re(wp^{-i})\right)+\sum_{i=k}^{\infty}s_i\Re(wp^{-i})\\
&\le -\Re (w)+m\sum_{i=1}^{k}\Re(wp^{-i})\\
&<0.
\end{align*}
Hence $\sum_{i=1}^{\infty}s_ip^{-i}\ne 0$.
\medskip

\emph{Step 4: $p$ is nonpositive and $\abs{p}= 1$.} We choose a complex number $w$ 
with $|w|=1$
according to Lemma \ref{l36} (b) and we repeat the construction of Step 3.

If $(s_i)$ has a last nonzero element $s_n$, then $n\ge\max\set{k,1}$, $\sum_{i=0}^{n}s_iq^{-i}=0$ and
\begin{equation*}
c:=\sum_{i=0}^{n}s_i\Re(wp^{-i})\le -\Re (w)+m\sum_{i=1}^{k}\Re(wp^{-i})< 0.
\end{equation*}

By distinguishing the cases where $(\arg p)/\pi$ is rational or irrational, we can find a sequence of integers $0=r_0<r_1<\cdots$ such that $r_{j+1}-r_j>n$ for all $j$, and each $p^{-r_j}$ is sufficiently close to $1$ so that
\begin{equation*}
\sum_{i=0}^{n}s_i\Re(wp^{-r_j-i})\le c/2,\quad j=1,2,\ldots .
\end{equation*}
Then replacing the sequence $(s_i)$ by
\begin{equation*}
s_0\ldots s_n 0^{r_1-n-1}s_0\ldots s_n 0^{r_2-r_1-n-1}s_0\ldots s_n 0^{r_3-r_2-n-1}\ldots 
\end{equation*}
we have 
\begin{equation*}
\sum_{i=0}^{\infty}s_iq^{-i}=0
\quad\text{and}\quad
\sum_{i=0}^{\infty}s_i\Re(wp^{-i})\le \sum_{j=0}^{\infty} c/2=-\infty,
\end{equation*}
so that the partial sums of the series $\sum s_ip^{-i}$ are unbounded.
\medskip

If the sequence $(s_i)$ is infinite, then we adapt the reasoning in \cite{ErdKom1998}, pp. 70--72 as follows. Assume that the partial sums
\begin{equation*}
S_r:=\sum_{i=0}^rs_iwp^{-i},\quad r=0,1,\ldots 
\end{equation*}
belong to finitely many circles centered at $0$. Then they are bounded. Let us also observe that
\begin{equation*}
\Re S_k=-\Re (w)+m\sum_{i=1}^{k}\Re(wp^{-i}) <0,
\end{equation*}
and
\begin{equation*}
\Re S_k\ge \Re S_{k+1}\ge \cdots
\end{equation*}
because $s_i\Re(wp^{-i})\le 0$ for all $i>k$ by construction. Moreover, by the infiniteness of $(s_i)$ and by Lemma \ref{l36} (b) we have $s_i\Re(wp^{-i})< 0$ for infinitely many indices $i>k$, and $\Re S_{i-1}>\Re S_i$ for all such indices.

Applying the Bolzano--Weierstrass theorem, there exists therefore a convergent subsequence $(S_{r_j})$ of $(S_r)$ with $r_1\ge k$, satisfying the following conditions:
\begin{align*}
&0>\Re S_{r_1}>\Re S_{r_2}>\cdots;\\
&\Re S_{r_j}\to\alpha<0;\\
&\Im S_{r_j}\to\beta;\\
&\abs{S_{r_j}}=\sqrt{\alpha^2+\beta^2}\quad\text{for all}\quad j.
\end{align*}
Using the definition of the tangent to a circle, we deduce from these properties that 
\begin{equation}\label{34}
\left|\frac{\Im (S_{r_{j+1}}-S_{r_j})}{\Re (S_{r_{j+1}}-S_{r_j})}\right|\to 
\begin{cases}
\abs{\alpha/\beta}>0&\text{if $\beta\ne 0$,}\\
\infty &\text{if $\beta=0$}
\end{cases}
\end{equation}
as $j\to\infty$.

On the other hand, we have for each index $i$ the inequalities
\begin{align*}
0\le \abs{s_iwp^{-i}}-\abs{\Im (s_iwp^{-i})}
&=\abs{s_i}\left( 1-\sqrt{1-\abs{\Re (wp^{-i})}^2}\right) \\
&\le \abs{s_i}\cdot\abs{\Re (wp^{-i})}^2\\
&\le \abs{s_i}^2\abs{\Re (wp^{-i})}^2\\
&= \abs{\Re (s_iwp^{-i})}^2
\end{align*}
because $1-\sqrt{1-x}\le x$ for all $0\le x\le 1$.

Since $\Re (s_iwp^{-i})\to 0$, for each fixed $\eps>0$ we have
\begin{equation*}
0\le \abs{s_i}-\abs{\Im (s_iwp^{-i})}=\abs{s_iwp^{-i}}-\abs{\Im (s_iwp^{-i})}\le \eps \abs{\Re (s_iwp^{-i})}
\end{equation*}
for all sufficiently large $i$. Hence for every sufficiently large $j$ the distance of
\begin{equation*}
\Im(S_{r_{j+1}}-S_{r_j})=\sum_{i=r_j+1}^{r_{j+1}} \Im (s_iwp^{-i})
\end{equation*}
from $\ZZ$ is at most
\begin{equation*}
\eps \sum_{i=r_j+1}^{r_{j+1}} \abs{\Re (s_iwp^{-i})}=\eps \abs{\Re (S_{r_{j+1}}-S_{r_j})}.
\end{equation*}
Since $\Im(S_{r_{j+1}}-S_{r_j})\to 0$, we conclude that
\begin{equation*}
\abs{\Im(S_{r_{j+1}}-S_{r_j})}\le \eps \abs{\Re (S_{r_{j+1}}-S_{r_j})}
\end{equation*}
for all sufficiently large $j$. 

Letting $\eps\to 0$ we obtain that
\begin{equation*}
\frac{\Im (S_{r_{j+1}}-S_{r_j})}{\Re (S_{r_{j+1}}-S_{r_j})}\to 0
\end{equation*}
as $j\to\infty$, contradicting \eqref{34}.
\end{proof}

\section{Proof of Theorem \ref{t13}}\label{s4}

%
%
If $q<m+1$ and $q$ is not Pisot, then $Y^m(q)$ has an accumulation point by Theorem \ref{t11}. Thus $0$ is an accumulation point of $Y^{2m}(q)=Y^m(q)-Y^m(q)$ (see \cite[Lemma 1.2]{ErdKom1998}) and hence
\begin{equation*}
\ell^{2m}(q)=\inf\set{\abs{y}\ :\ y\in Y^{2m}(q), y\ne 0}=0.
\end{equation*}

Next we apply Lemma \ref{l24} with $(u_k)=(x^{2m}_k(q))$, $(v_k)=(x^m_k(q))$ and $(z_k)=(x^{3m}_k(q))$. The first condition of this lemma is satisfied by the just proven equality $\ell^{2m}(q)=0$ and by applying Lemma \ref{l22} with $2m$ in place of $m$, while the second condition is satisfied by Lemma \ref{l21} (i). We conclude that $L^{3m}(q)=0$.

\section{Proof of Theorem \ref{t14}}
\label{s5}

First we recall  \cite[Lemma 2.2]{ErdKom1998}:

\begin{lemma}\label{l51}
Let $(u_k),(v_k),(z_k)$ be three increasing sequences of real numbers such that each sum $u_k+v_{\ell}$ appears in $(z_n)$. Assume that

\begin{itemize}
\item the set $\set{u_k-u_{\ell}\ :\ k,\ell=0,1,\ldots}$ has an accumulation point;
\item $v_k\to\infty$ but the difference sequence $(v_{k+1}-v_k)$ is bounded from above.
\end{itemize}
Then $\liminf (z_{k+1}-z_k)=0$.
\end{lemma}

We also need two new lemmas.

\begin{lemma}\label{l52}
If $Y^m(q^2)$ has an accumulation point, then $\ell^m(q)=0$. 
\end{lemma}

\begin{proof}
We have $q^2<m+1$ by  (an easy part of) Theorem \ref{t11}. We apply Lemma \ref{l51} with $(u_k)=(x^m_k(q^2))$, $(v_k)=(qx^m_k(q^2))$ and $(z_k)=(x^m_k(q))$. The hypothesis on $(u_k)$ is satisfied by our assumption on $Y^m(q^2)$, while the hypothesis on $(v_k)$ is satisfied by using Lemma \ref{l21} (i) with $q^2$ instead of $q$. We conclude that $\ell^m(q)=0$. 
\end{proof}

\begin{lemma}\label{l53}
If $Y^m(q^3)$ has an accumulation point, then $L^m(q)=0$. 
\end{lemma}

\begin{proof}
The first part of the proof is similar to the preceding proof. We have $q^3<m+1$ by  Theorem \ref{t11}. Applying Lemma \ref{l51} with $(u_k)=(x^m_k(q^3))$ and $(v_k)=(qx^m_k(q^3))$ we obtain that the increasing sequence $(z_k)$ formed by the sums $u_p+v_r$ satisfies the condition $\liminf (z_{k+1}-z_k)=0$. 

Now in view of Lemma \ref{l22} and Remark \ref{r23} we may apply Lemma \ref{l24} with this sequence $(z_k)$ in place of $(u_k)$ and with $(v_k)=(q^2x^m_k(q^3))$ and $(z_k)=(x^m_k(q))$. We conclude that $L^m(q)=0$. 
\end{proof}

Finally we recall the following theorem of Sidorov--Solomyak \cite[Theorems 2.1 and 2.4]{SidSol2009} shown by using the connectedness of the associated fractal set:

\begin{theorem}\label{t54}
If $\ell^1(q)>0$ for some $1<q<2$, then $q$ is an algebraic integer and its conjugates $p$ satisfy the relations
\begin{equation*}
q^{-1}<\abs{p}<q\quad\text{or}\quad p=\pm q^{-1}.
\end{equation*}
\end{theorem}
\noindent
This is an improvement of an earlier theorem of Stankov \cite{Sta2010}.

\begin{proof}[Proof of Theorem \ref{t14}]\mbox{}

(i) The case $q=2^{1/3}$ follows by replacing $\sqrt{2}$ by $2^{1/3}$ in the proof of \cite[Proposition 9]{ErdJooKom1998}. Henceforth we assume that  $1<q<2^{1/3}$.

If $q^3$ is not a Pisot number, then the equality $L^1(q)=0$ follows from Theorem \ref{t11} and Lemma \ref{l53}.

Assume that $q^3$ is a Pisot number.
Since $q< 1.26$ and thus it is smaller than the first Pisot number $\approx 1.32$ (see \cite{PisotSalem}), 
$q$ is not a Pisot number. Applying Lemma \ref{l27} we conclude that  $q^2$ is not a Pisot number. 

Thus $q^2$ is not a Pisot number but $q^6$ is a Pisot number. Consider a  conjugate $p$ of $q^2$ satisfying $\abs{p}\ge 1$. Then $p^3=q^6$. For otherwise $p^3$ were a conjugate of $q^6$, contradicting the Pisot property of $q^6$ because $\abs{p^3}\ge 1$. It follows that $p=q^2\zeta_3$ or $p=q^2\zeta_3^2$ and hence $\abs{p}=q^2$. 
Applying  Theorem \ref{t54} we obtain that $\ell^1(q^2)=0$. 

Applying Lemma \ref{l25} we conclude that $L^1(q)=0$.
\medskip

(ii) First we prove that $\ell^1 (q)=0$. If $q$ is transcendental, then we may apply, e.g., \cite[Theorem 5]{ErdJooKom1998}. 

More generally, if $q^2$ is not a Pisot number, then the equality $\ell^1 (q)=0$ follows from Theorem \ref{t11} and Lemma \ref{l52}.

If $q^2$ is a Pisot number, then $\abs{p}<1$ for all conjugates of $q$ other than $-q$, because then $p^2$ is a conjugate of $q^2$ and hence $\abs{p^2}<1$. Since $q$ is not a Pisot number by assumption, $-q$ is a conjugate of $q$. Applying   Theorem \ref{t54} we conclude again that $\ell^1 (q)=0$. 
\medskip

Next we prove that $L^2(q)=0$ . We already know that $\ell^1 (q)=0$. Thanks to Lemma \ref{l22}  we may apply Lemma \ref{l24} with $(u_k)=(v_k)=(x^1_k(q))$ and $(z_k)=(x^2_k(q))$. We conclude that $L^2(q)=0$.
\medskip

(iii) As we have already mentioned in the introduction, this is a special case of Theorem \ref{t13}.
\end{proof}

\section{Proof of Theorem \ref{t15}}
\label{s6}

The difference between Theorem \ref{t15} and \cite[Proposition 1.2]{PerSol2000} is that instead of assuming that $q$ is non-Pisot, it was assumed that $q^2$ is not a root of a polynomial with coefficients in $\set{-1,0,1}$. This latter condition was only used to ensure (via \cite[Lemma 8]{ErdJooKom1998}) that $\ell^1 (q)=0$. Hence we can repeat the proof given in \cite{PerSol2000} by applying our Theorem \ref{t14} (ii): since $1<q\le\sqrt{2}$ is non-Pisot, we have $\ell^1 (q)=0$.

\section{Open problems}
\label{s7}

We conclude our paper by formulating some interesting open problems.

\begin{enumerate}
\item An old question (see \cite{ErdJooKom1998}) is whether $\ell^1(q)=0$ for all non-Pisot numbers $1<q<2$. 

\item Another old question (see \cite{Kom2002b}) is whether $L^1(q)=0$ for all non-Pisot numbers $1<q<(1+\sqrt{5})/2$.

\item More generally, characterize the inequalities $\ell^m(q)>0$ and $L^m(q)>0$ in a way analogous to Theorem \ref{t11}. 

\item Characterize  the following four classes for $Y^m(q)$:
\begin{enumerate}
\item[(i)] $Y^m(q)$ is dense in $\RR$;
\item[(ii)] $Y^m(q)$ is uniformly discrete in $\RR$;
\item[(iii)] $Y^m(q)$ is discrete closed, but not uniformly discrete in $\RR$;
\item[(iv)] none of the above.
\end{enumerate}

\end{enumerate}

\begin{examples}\label{e72}\mbox{}
We recall from Drobot \cite{Dro1973} that $\ell^m(q)=0$ is equivalent to the denseness of $Y^m(q)$, and from  \eqref{12} and the equality $Y^{2m}(q)=Y^m(q)-Y^m(q)$ that uniform discreteness of $Y^m(q)$  is equivalent to $\ell^{2m}(q)>0$. In view of this, 
the cases (i)-(iii) above all occur by the following consequences of Theorems \ref{t11}, \ref{t13} and Remark \ref{r12}:

\begin{itemize}
\item[(i)] $Y^m(q)$ is dense in $\RR$ for all but countably many $1<q<m+1$;
\item[(ii)] $Y^m(q)$ is uniformly discrete if $q$ is a Pisot number or if $q\ge 2m+1$;
\item[(iii)] $Y^m(q)$ is discrete closed for all $q>m+1$, but it is uniformly discrete only for countably many  $m+1\le q\le 2m+1$.
\end{itemize}

The case (iv) is not expected to occur. In view of Theorem \ref{t11} this amounts to prove that $Y^m(q)$ is dense in $\RR$  for all non-Pisot numbers $1<q<m+1$ or equivalently, that the existence of a finite accumulation point of  $Y^m(q)$ implies that $0$ is also an accumulation point.
\end{examples}

\end{document}